\documentclass[11pt]{amsart}
\usepackage{amssymb,amsmath,amsthm}
\usepackage{amsfonts}
\usepackage{enumerate}
\usepackage{dsfont}
\usepackage{cite}
\usepackage[colorlinks, citecolor=red]{hyperref}
\usepackage{mathrsfs}
\usepackage{epsfig}
\usepackage{lscape}
\usepackage{subfigure}
\usepackage{epstopdf}
\usepackage{caption}
\usepackage{algorithm}
\usepackage{algpseudocode}
\usepackage{multirow}
\usepackage{geometry}
\usepackage{paralist}
\usepackage{enumerate}
\usepackage{url}
\usepackage{graphicx,cite}
\usepackage{longtable}
\usepackage{tikz}

\newtheorem{definition}{Definition}[section]
\newtheorem{corollary}[definition]{Corollary}

\newtheorem{theorem}[definition]{Theorem}

\newtheorem{remark}[definition]{Remark}

\date{}

\begin{document}
\baselineskip 18pt
\bibliographystyle{plain}
\title[RKAS for inconsistent systems]
{Randomized Kaczmarz method with adaptive stepsizes for inconsistent linear systems}

\author{Yun Zeng}
\address{School of Mathematical Sciences, Beihang University, Beijing, 100191, China. }
\email{zengyun@buaa.edu.cn}

\author{Deren Han}
\address{LMIB of the Ministry of Education, School of Mathematical Sciences, Beihang University, Beijing, 100191, China. }
\email{handr@buaa.edu.cn}

\author{Yansheng Su}
\address{School of Mathematical Sciences, Beihang University, Beijing, 100191, China. }
\email{suyansheng@buaa.edu.cn}

\author{Jiaxin Xie}
\address{LMIB of the Ministry of Education, School of Mathematical Sciences, Beihang University, Beijing, 100191, China. }
\email{xiejx@buaa.edu.cn}

\begin{abstract}
We investigate the randomized Kaczmarz method that adaptively updates the stepsize using readily available information  for solving  inconsistent linear systems.
A novel geometric interpretation is provided which shows that the proposed method  can be viewed as an orthogonal projection method in some sense.
We prove that this method converges linearly in expectation to the unique minimum Euclidean norm least-squares solution of the linear system, and provide a tight upper bound for the convergence of the proposed method. Numerical experiments are also given to illustrate the theoretical results.
\end{abstract}

\maketitle

\let\thefootnote\relax

\let\thefootnote\relax

\section{Introduction}

Solving systems of linear equations is a fundamental problem in scientific computing  and engineering. It comes up in many real-world applications such as signal processing \cite{Byr04},
optimal control \cite{Pat17}, machine learning \cite{Cha08},  and partial differential equations \cite{Ols14}.
The \emph{Kaczmarz} method \cite{Kac37}, also known as \emph{algebraic reconstruction technique} (ART) \cite{herman1993algebraic,gordon1970algebraic}, is a classic yet effective  row-action iteration solver for solving the large-scale linear system of equations
\begin{equation}
\label{main-prob}
Ax=b, \ A\in\mathbb{R}^{m\times n}, \ b\in\mathbb{R}^m.
\end{equation}
At each step of the original Kaczmarz method, a row of the system is sampled and the previous iterate is orthogonally projected onto the hyperplane defined by that row.

In the literature, there are empirical evidences that using the rows of the matrix $A$  in a random order rather than a deterministic order can often accelerate the convergence of the Kaczmarz method \cite{herman1993algebraic,natterer2001mathematics,feichtinger1992new}.
In the seminal paper \cite{Str09}, Strohmer and Vershynin studied the \emph{randomized  Kaczmarz} (RK) method and proved its linear convergence in expectation provided that the  linear system \eqref{main-prob}
is \emph{consistent}. Subsequently, there is a large amount of work on the development of the Kaczmarz-type methods including accelerated randomized Kaczmarz methods \cite{liu2016accelerated,han2022pseudoinverse,loizou2020momentum}, block Kaczmarz methods \cite{Nec19,needell2014paved,moorman2021randomized,Gow15}, greedy randomized Kaczmarz methods \cite{Bai18Gre,Gow19}, randomized sparse Kaczmarz methods \cite{schopfer2019linear,chen2021regularized}, etc. Nevertheless, all of these methods will not converge if the linear system \eqref{main-prob} is not consistent.
Indeed,  Needell \cite{needell2010randomized} showed that RK applied to \emph{inconsistent} linear systems converges only to within a radius (\emph{convergence horizon}) of the least-squares solution (see Theorem \ref{thm-inc}); see also \cite{bai2021greedy} for some further comments.

It is well-known that  the so-called relaxation parameters or stepsizes $\lambda_1,\ldots,\lambda_k$ are important for the Kaczmarz method in practice. The original Kaczmarz method with decreasing stepsizes for solving \emph{inconsistent} systems has been investigated in \cite{censor1983strong,hanke1990acceleration}.
It has been shown that with the stepsizes being nearly zero and appropriate initial point, the Kaczmarz method converges inside the convergence horizon to the minimum Euclidean norm least-squares solution.
However, its convergence rate is difficult to obtain. Hence, for the RK method, a natural and interesting question is that is it possible that with carefully designed stepsizes the  RK method is convergent for solving inconsistent systems? Furthermore, can the convergence rate of the proposed method be obtained easily?


Actually,  the \emph{randomized extended Kaczmarz} (REK) method \cite{Zou12,Du19} has already provided a positive answer to the above questions. Section \ref{section-geo} will provide more detailed comments on this topic.
Furthermore, there is enormous of work on the developments and extensions of the REK method, including the block or deterministic variants of REK \cite{wu2021semiconvergence,needell2013two,DS2021,Du20Ran,wu2022two,wu2022extended,popa1998extensions,popa1999characterization,bai2019partially}, \emph{greedy randomized augmented Kaczmarz} (GRAK) method \cite{bai2021greedy}, \emph{randomized extended Gauss-Seidel} (REGS) method \cite{Du19,ma2015convergence}, etc.
We note that those methods make use of both rows and columns of $A$ at each step (see \eqref{rek}) and work for general linear systems (consistent or inconsistent, full-rank or rank-deficient).
Another randomized method that can be used to solve inconsistent systems with full column-rank coefficient matrix is the \emph{randomized coordinate descent} (RCD) method \cite{Lev10}, we refer to \cite{bai2021convergence} for more discussions about the RCD method.


In this paper, we further investigate the RK method with adaptive stepsizes for solving inconsistent systems and provide an alternative strategy to answer the above questions.
Our proposed geometric interpretation demonstrates that the method can be viewed as an orthogonal projection method in some sense.
By utilizing this interpretation, we show that our strategy is effective in simplifying the analysis  and endows the proposed method with a linear convergence rate.
Additionally, we conduct a comprehensive comparison between the proposed method and the REK method, including their geometric interpretation, theoretical analysis, and numerical behavior.


 The remainder of the paper is organized as follows.
 After introducing some preliminaries in Section 2, we present and analyze the RK method with adaptive stepsizes in Section 3.
In Section 4, we perform some numerical experiments to show
the effectiveness of the proposed method. Finally, we conclude the paper in Section 5.

\section{Preliminaries}
\subsection{Notations}
Throughout the paper, for any random variables $\xi$ and $\zeta$, we use $\mathbb{E}[\xi]$ and $\mathbb{E}[\xi\lvert \zeta]$ to denote the expectation of $\xi$ and the conditional expectation of $\xi$ given $\zeta$. For an integer $m\geq 1$, let $[m]:=\{1,\ldots,m\}$.
Given $T\subset[m]$, the cardinality
of the set $T$ is denoted by $\mid T\mid$.
For any vector $x\in\mathbb{R}^n$, we use $x_i,x^\top$, and $\|x\|_2$ to denote the $i$-th entry, the transpose and the Euclidean norm of $x$, respectively.
For any matrix $A\in\mathbb{R}^{m\times n}$, we use $A_{i,:}, A_{:,j},A^\top,A^\dagger,\|A\|_2,\|A\|_F$, $\mbox{Range}(A)$,  and $\text{Null}(A)$ to denote the $i$-th row, the $j$-th column, the transpose, the Moore-Penrose pseudoinverse, the spectral norm, the Frobenius norm, the column space, and the null space of $A$, respectively.
For any $x,y\in\mathbb{R}^n$, the Hadamard product of $x$ and $y$ is defined to be the entrywise product $(x\circ y)_i=x_iy_i$.
The nonzero singular values of a matrix $A$ are $\sigma_1(A)\geq\sigma_2(A)\geq\ldots\geq\sigma_{r}(A):=\sigma_{\min}(A)>0$, where $r$ is the rank of $A$
and $\sigma_{\min}(A)$ denotes the smallest nonzero singular values of $A$.  We see that $\|A\|_2=\sigma_{1}(A)$ and $\|A\|_F=\sqrt{\sum\limits_{i=1}^r \sigma_i(A)^2}$.

\subsection{The pseudoinverse solution}
In this paper, we are interested in the  pseudoinverse solution $A^\dagger b$ of the linear system \eqref{main-prob}. Here we would like to make clear what $A^\dagger b $ represents in different cases of linear systems \cite{golub2013matrix,ben2003generalized,Du20Ran}. Table \ref{table11} summarizes the results.

\begin{table}
\setlength{\tabcolsep}{2pt}
\caption{The pseudoinverse solution $A^{\dagger} b$ of $Ax=b$.}
\label{table11}
\centering
\begin{tabular}{c|c|c}
\hline

\hline
$A x=b$ & $\operatorname{rank}(A)$ & $A^{\dagger} b$ \\
\hline consistent & $=n$ & unique solution \\
consistent & $<n$ & unique minimum Euclidean norm solution \\
inconsistent & $=n$ & unique least-squares (LS) solution \\
inconsistent & $<n$ & unique minimum Euclidean norm LS solution \\
\hline

\hline
\end{tabular}
\end{table}

\subsection{The RK method}
The RK method for solving the linear system \eqref{main-prob} begins with an arbitrary vector $x^0$, and in the $k$-th iteration iterates by
\begin{equation}\label{RK-method}
x^{k+1}=x^k-\lambda\frac{ A_{i_k,:}x^k-b_{i_k}}{\|A_{i_k,:}\|^2_2}A_{i_k,:}^\top,
\end{equation}
where the index $i_k$ is i.i.d. selected from $[m]$ and $\lambda\in(0,2)$ is the stepsize. When $\lambda=1$, it reduces to the classical RK method. In the seminal paper \cite{Str09}, Strohmer and Vershynin proved the first linear convergence rate of the RK method for consistent systems. Later, Needell \cite{needell2010randomized,needell2014stochasticMP} studied the RK method for inconsistent cases. The result is precisely restated below.
\begin{theorem}[\cite{needell2014stochasticMP}, Corollary 5.1]
\label{thm-inc}
Starting from any initial vector $x^0\in\text{Range}(A^\top)$, the expected error  of the RK method \eqref{RK-method} in the $k$-th iteration satisfies
\begin{equation}\label{conv-RK}
\mathbb{E}\left[\left\|x^k-A^\dagger b\right\|^2_2\right]\leq\left(1-2\lambda(1-\lambda)\frac{\sigma^2_{\min}(A)}{\|A\|^2_F}\right)^k\left\|x^0-A^\dagger b\right\|^2_2+
\frac{\lambda a^2_{\max}\|e\|^2_2}{(1-\lambda)a^2_{\min}\sigma^2_{\min}(A)},
\end{equation}
where the index $i$ is selected with probability $\frac{\|A_{i,:}\|^2_2}{\|A\|^2_F}$,
 $\lambda\in(0,1)$, $e=AA^\dagger b-b$, $a^2_{\min}=\min\limits_{i\in[m]}\|A_{i,:}\|^2_2$, and $a^2_{\max}=\max\limits_{i\in[m]}\|A_{i,:}\|^2_2$.
\end{theorem}
It can be seen that for arbitrary $\lambda$,
the above result implies a tradeoff between a smaller convergence horizon and a slower
convergence. In this paper, we study the RK method with adaptive stepsizes $\lambda_k$ to eliminate the last term in \eqref{conv-RK} (see Corollary \ref{corol-1}).

\section{Adaptive stepsizes for RK }
\label{section-RrDR}

In this section, we  introduce \emph{RK with adaptive stepsizes} (RKAS) for solving the linear system \eqref{main-prob}.
The method is formally described in Algorithm \ref{rkas}.

\begin{algorithm}[htpb]
\caption{RK with adaptive stepsizes (RKAS) \label{rkas}}
\begin{algorithmic}
\Require
$A\in \mathbb{R}^{m\times n}$, $b\in \mathbb{R}^m$, $k=0$ and initial points $x^0=0, r^0=Ax^0-b=-b$.
\begin{enumerate}
\item[1:] Select $i_{k}\in[m]$ with probability $\mbox{Pr}(i_k=i)=\frac{\|A_{i,:}\|^2_2}{\|A\|_{F}^2}$.
\item[2:] Compute
$$
\alpha_k=\frac{\langle AA_{i_k,:}^\top,r^k\rangle}{\|AA_{i_k,:}^\top\|^2_2}.
$$
\item[3:] Update
\begin{align*}
x^{k+1}&=x^k-\alpha_kA_{i_k,:}^\top,
\\
r^{k+1}&=r^k-\alpha_k A A_{i_k,:}^\top.
\end{align*}
\item[4:] If the stopping rule is satisfied, stop and go to output. Otherwise, set $k=k+1$ and return to Step $1$.
\end{enumerate}

\Ensure
  The approximate solution.
\end{algorithmic}
\end{algorithm}

Note that the most expensive computational cost in the $k$-th iteration of Algorithm \ref{rkas} is to compute $A A_{i_k,:}^\top$. Let $B:=AA^\top$, then $B_{:,i_k}=A A_{i_k,:}^\top$. Thus, if it is possible to store $B=AA^\top$ at the initialization, Algorithm \ref{rkas} could be faster in practice. In fact, this strategy is also adopted by the greedy randomized Kaczmarz method \cite{bai2021greedy,Bai18Gre} and the weighted randomized Kaczmarz method \cite{Ste20Wei}.

\subsection{A geometric interpretation}
We  present an intuitive geometric explanation of Algorithm \ref{rkas} in this subsection.
Consider the following least-squares problem
$$\min\limits_{x\in\mathbb{R}^n} f(x):=\frac{1}{2}\|Ax-b\|^2_2.$$
Since
\begin{align*}
f(x)&=\frac{1}{2}x^\top A^\top A x-x^\top A^\top b+\frac{1}{2}\|b\|_2^2\\
&=\frac{1}{2}x^\top A^\top A x-x^\top A^\top A A^\dagger b+\frac{1}{2}\|b\|_2^2
\\
&=\frac{1}{2}\|Ax-AA^\dagger b\|^2_2-\frac{1}{2}\|AA^\dagger b\|^2_2+\frac{1}{2}\|b\|_2^2,
\end{align*}
where the second equality follows from the fact that $A^\top A A^\dagger b=A^\top b$.
This implies that the least-squares problem can be equivalently reformulated as
\begin{equation}\label{LS-prob}
\min\limits_{x\in\mathbb{R}^n}\frac{1}{2}\|Ax-AA^\dagger b\|^2_2.
\end{equation}
When using RK \eqref{RK-method} to solve the least-squares problem \eqref{LS-prob}, in the $k$-th iteration, we may expect the distance between $Ax^{k+1}$ and $AA^\dagger b$ to be as small as possible. This leads to the following optimization problem:
\begin{equation}\label{opt-prob}
{\boxed{
x^{k+1}=\arg\min\limits_{x\in\mathbb{R}^n} \|Ax-AA^\dagger b\|^2_2 \ \ \text{subject to} \ \ x=x^k-\lambda\frac{A_{i_k,:}x^k-b_{i_k}}{\|A_{i_k,:}\|^2_2}A_{i_k,:}^\top, \ \lambda\in\mathbb{R}.
}}\end{equation}
Note that $r^{k+1}$ in Step $3$ is actually obtained  by  an incremental method
$$
r^{k+1}=r^k-\alpha_k A A_{i_k,:}^\top=Ax^{k}-b-\alpha_k A A_{i_k,:}^\top=Ax^{k+1}-b.
$$
Using the fact that $A^\top A A^\dagger b=A^\top b$ and letting $\gamma_k=\frac{A_{i_k,:}x^k-b_{i_k}}{\|A_{i_k,:}\|^2_2}$, then the minimizer of \eqref{opt-prob} is achieved when
$$\lambda_k^*=\frac{\langle AA_{i_k,:}^\top,Ax^k-AA^\dagger b\rangle}{\|AA_{i_k,:}^\top\|^2_2\gamma_k}=\frac{\langle AA_{i_k,:}^\top,Ax^k-b\rangle}{\|AA_{i_k,:}^\top\|^2_2\gamma_k}=\frac{\langle AA_{i_k,:}^\top,r^k\rangle}{\|AA_{i_k,:}^\top\|^2_2\gamma_k}.$$
Hence
$$
x^{k+1}=x^k-\lambda^*_k\gamma_kA_{i_k,:}=x^k-\frac{\langle AA_{i_k,:}^\top,r^k\rangle}{\|AA_{i_k,:}^\top\|^2_2}A_{i_k,:}^\top=x^k-\alpha_kA_{i_k,:}^\top,
$$
which is exactly the iteration in Step 3 of Algorithm \ref{rkas}.
It follows from \eqref{opt-prob} that $x^{k+1}$ obtained by Algorithm \ref{rkas} satisfies that $Ax^{k+1}$ is the orthogonal projection of $AA^\dag b$ onto $Ax^k+\text{Span}\{AA_{i_k,:}^\top\}$.
 The geometric interpretation of Algorithm \ref{rkas} is presented in Figure \ref{figueRKAS}.

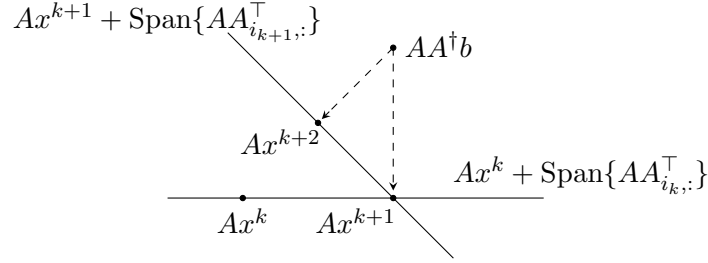
\begin{figure}[hptb]
	\centering
	\begin{tikzpicture}
		\draw (-2,0) -- (3,0);
		\draw (-1.2,2.2) -- (1.8,-0.8);
		
		\filldraw (1,2) circle [radius=1pt]
		(1,0) circle [radius=1pt]
        (-1,0) circle [radius=1pt]
        (0,1) circle [radius=1pt];
		
		\draw (3.5,0.3) node {$Ax^{k}+\mbox{Span}\{AA_{i_k,:}^\top\}$};
        \draw (-2,2.34) node {$Ax^{k+1}+\mbox{Span}\{AA_{i_{k+1},:}^\top\}$};
        \draw (1.6,2.05) node {$AA^\dagger b$};
        \draw (0.5,-0.3) node {$Ax^{k+1}$}
        (-1,-0.3) node {$Ax^k$}
		(-0.5,0.75) node {$Ax^{k+2}$};
		
		\draw [dashed,-stealth] (1,2) -- (1,0.1);
        \draw [dashed] (1,2) -- (0.1,1.1);
        \draw [dashed,-stealth](0.15,1.15) -- (0.05,1.05);
	\end{tikzpicture}
	\caption{A geometric interpretation of Algorithm \ref{rkas}. The next iterate $x^{k+1}$ arises such that $Ax^{k+1}$ is the projection of $AA^\dagger b$ onto $Ax^{k}+\mbox{Span}\{AA_{i_k,:}^\top\}$. }
	\label{figueRKAS}
\end{figure}

\subsection{Convergence analysis}
We now state our convergence results. The following result is about the convergence for  $\mathbb{E}[\|Ax^k-AA^\dagger b\|^2_2]$.
\begin{theorem}
\label{thm-1}
For any given linear system $Ax=b$, let $\{x^k\}_{k=0}^{\infty}$ be the iteration sequence generated by Algorithm \ref{rkas} with $x^0=0$. Then
$$
\mathbb{E}\left[\left\|Ax^k-AA^\dagger b\right\|^2_2\right]\leq \left(1-\frac{\sigma_{\min}^4(A)}{\|A\|_2^2\|A\|_F^2}\right)^k\left\|Ax^0-AA^\dagger b\right\|^2_2.
$$
\end{theorem}
\begin{proof}
Since $Ax^{k+1}-AA^\dagger b$ and $AA_{i_k,:}^\top$ are orthogonal, we have
\begin{align*}
\left\|Ax^{k+1}-AA^\dagger b\right\|^2_2&=\left\|Ax^k-AA^\dagger b\right\|^2_2-\left\|Ax^{k+1}-Ax^{k}\right\|^2_2\\
&=\left\|Ax^k-AA^\dagger b\right\|^2_2-\frac{\langle AA_{i_k,:}^\top,r^k\rangle^2}{\|AA_{i_k,:}^\top\|^2_2}\\
&=\left\|Ax^k-AA^\dagger b\right\|^2_2-\frac{\langle AA_{i_k,:}^\top,Ax^k-b\rangle^2}{\|AA_{i_k,:}^\top\|^2_2}\\
&=\left\|Ax^k-AA^\dagger b\right\|^2_2-\frac{\langle AA_{i_k,:}^\top,Ax^k-AA^\dagger b\rangle^2}{\|AA_{i_k,:}^\top\|^2_2},
\end{align*}
where the last equality follows from $A^\top b=A^\top AA^\dagger b$.
Hence
\begin{equation}\label{xie-0212}
\begin{aligned}
\mathbb{E}\left[\left\|Ax^{k+1}-AA^\dagger b\right\|^2_2 \bigg | x^k \right]&=\left\|Ax^k-AA^\dagger b\right\|^2_2-\sum\limits_{i=1}^m\frac{\|A_{i,:}\|^2_2}{\|A\|^2_F}\frac{\langle AA_{i,:}^\top,Ax_k-AA^\dagger b\rangle^2}{\|AA_{i,:}^\top\|^2_2}\\
&\leq \left\|Ax^k-AA^\dagger b\right\|^2_2-\sum\limits_{i=1}^m\frac{\|A_{i,:}\|^2_2}{\|A\|^2_F}\frac{\langle AA_{i,:}^\top,Ax^k-AA^\dagger b\rangle^2}{\|A\|^2_2\|A_{i,:}\|^2_2}
\\
&= \left\|Ax^k-AA^\dagger b\right\|^2_2-\sum\limits_{i=1}^m\frac{\langle AA_{i,:}^\top,Ax^k-AA^\dagger b\rangle^2}{\|A\|^2_2\|A\|^2_F}
\\
&=\left\|Ax^k-AA^\dagger b\right\|^2_2-\frac{\|AA^\top(Ax^k-AA^\dagger b)\|_2^2}{\|A\|^2_2\|A\|^2_F}\\
&\leq \left(1-\frac{\sigma^4_{\min}(A)}{\|A\|_2^2\|A\|^2_F}\right)\left\|Ax^k-AA^\dagger b\right\|^2_2,
\end{aligned}
\end{equation}
where the first inequality follows from $\|AA_{i,:}^\top\|_2\leq \|A\|_2\|A_{i,:}\|_2$ and the last inequality follows from the fact that $A(x^k-A^\dagger b)\in\text{Range}(AA^\top)$.
Taking expectation over the entire history we have
	$$
\mathbb{E}\left[\left\|Ax^{k+1}-AA^\dagger b\right\|^2_2 \right]\leq \left(1-\frac{\sigma^4_{\min}(A)}{\|A\|_2^2\|A\|^2_F}\right)\mathbb{E}\left[\left\|Ax^k-AA^\dagger b\right\|^2_2\right].
	$$
	By induction on the iteration index $k$, we can obtain the desired result.
\end{proof}

By Theorem \ref{thm-1}, we can obtain the following  linear convergence for the expected norm of the error.

\begin{corollary}\label{corol-1}
For any given linear system $Ax=b$, let $\{x^k\}_{k=0}^{\infty}$ be the iteration sequence generated by Algorithm \ref{rkas} with $x^0=0$. Then
$$
\mathbb{E}\left[\left\|x^k-A^\dagger b\right\|^2_2\right]\leq \sigma^{-2}_{\min}(A)\left(1-\frac{\sigma_{\min}^4(A)}{\|A\|_2^2\|A\|_F^2}\right)^k\left\|Ax^0-AA^\dagger b\right\|^2_2.
$$
\end{corollary}

\begin{proof}
According to the iteration of Algorithm \ref{rkas} and $x^0=0$, we know that $x^k\in\text{Range}(A^\top)$. As $A^\dagger b\in\text{Range}(A^\top)$, we have $x^k-A^\dagger b\in\text{Range}(A^\top)$. This implies
\begin{equation}\label{xie-0212-1}
\left\|Ax^k-AA^\dagger b\right\|^2_2\geq \sigma^2_{\min}(A)\left\|x^k-A^\dagger b\right\|^2_2.
\end{equation}
Then by Theorem \ref{thm-1}, we arrive at this corollary.
\end{proof}

\begin{remark}
If $\sigma_1(A) = \sigma_{\min}(A)$, that is, all nonzero singular values of $A$ are equal, then the inequalities in \eqref{xie-0212} and \eqref{xie-0212-1} become equalities. Consequently, the upper bounds in Theorem \ref{thm-1} and Corollary \ref{corol-1} are also equalities, indicating that the upper bounds in Theorem \ref{thm-1} and Corollary \ref{corol-1} are tight.
\end{remark}

\begin{remark}
We note that one can choose any $x^0\in\mathbb{R}^n$ as the initial vector. In this case, we can prove that
$$
\mathbb{E}\left[\left\|x^k-x^0_{*}\right\|^2_2\right]\leq \sigma^{-2}_{\min}(A)\left(1-\frac{\sigma_{\min}^4(A)}{\|A\|_2^2\|A\|_F^2}\right)^k\left\|Ax^0-Ax^0_{*}\right\|^2_2,
$$
where $x^0_*:=A^\dagger b+(I-A^\dagger A)x^0$. We refer the reader to \cite{DS2021,HSXDR2022,han2022pseudoinverse} for more details.
\end{remark}

\begin{remark}\label{remark2}
For the RK method \eqref{RK-method}, from \eqref{conv-RK} and  with an analysis analogous to \cite[Corollary 2.2]{needell2014stochasticMP}, for any desired $\varepsilon$, using a stepsize
 $$
 \lambda=\frac{\varepsilon\sigma^2_{\min}(A)a^2_{\min}}{2\varepsilon\sigma^2_{\min}(A)\alpha^2_{\min}+2 \|e\|_2^2 a^2_{\max}},
 $$
 one has that after
 \begin{equation}\label{iter-rk}
 k=2\log\left(2\varepsilon_0/\varepsilon\right)\left(\frac{\|A\|^2_F}{\sigma^2_{\min}(A)}+\frac{\|A\|^2_F \|e\|^2_2 a^2_{\max}}{\varepsilon\sigma^4_{\min}(A)a^2_{\min}}\right)
 \end{equation}
 iterations, $\mathbb{E}[\|x^k-A^\dagger b\|^2_2]\leq \varepsilon$, where $\varepsilon_0=\|x^0-A^\dagger b\|^2_2$. For Algorithm \ref{rkas}, according to corollary \ref{corol-1}, we have that after
 \begin{equation}\label{iter-rkas}
 k=\log\left(\frac{\varepsilon_1}{\varepsilon\sigma^2_{\min}(A)}\right)\frac{\|A\|^2_F\|A\|^2_2}{\sigma^4_{\min}(A)}
\end{equation}
 iterations,  $\mathbb{E}[\|x^k-A^\dagger b\|^2_2]\leq \varepsilon$, where $\varepsilon_1=\|Ax^0-AA^\dagger b\|^2_2$. From \eqref{iter-rk} and \eqref{iter-rkas}, we know that the RKAS method shall use less number of iterations than that of the RK method to obtain an iterative solution with the accuracy $\varepsilon<O(1/\|A\|^2_2)$.
\end{remark}

\subsection{The relationship between REK and RKAS}
\label{section-geo}

Recently, the randomized extended Kaczmarz (REK) method \cite{Zou12,Du19} has attracted much attention for solving inconsistent systems. The method generates two sequences $\{\tilde{z}^{k}\}_{k=0}^\infty$ and $\{\tilde{x}^{k}\}_{k=0}^\infty$ via
\begin{equation}\label{rek}
	\begin{array}{ll}
		\tilde{z}^{k+1}&=\tilde{z}^{k}-\frac{A^{\top}_{:,j_k}\tilde{z}^{k}}{\|A_{:,j_k}\|^2_2}A_{:,j_k},
		\\
		\tilde{x}^{k+1}&=\tilde{x}^{k}-\frac{ A_{i_k,:} \tilde{x}^k-b_{i_k}+(\tilde{z}^{k+1})_{i_k}}{\|A_{i_k,:}\|^2_2}A_{i_k,:}^\top,
	\end{array}
\end{equation}
where the column $A_{:,j_k}$ is chosen with probability $\frac{\|A_{:,j_k}\|^2_2}{\|A\|^2_F}$ and the row $A_{i_k,:}$ is chosen with probability $\frac{\|A_{i_k,:}\|^2_2}{\|A\|^2_F}$, see \cite{Du19}.

Firstly, let us recall the geometric interpretation of the REK method discussed in previous works \cite{Zou12,Du19}.
We  define the hyperplanes  as follows
$$H_{i_k}=\{x| A_{i_k, :} x=(AA^{\dagger}b)_{i_k}\} \ \text{and} \ H_{i_k,\tilde{z}^{k+1}}=\{x|A_{i_k, :} x=(b-\tilde{z}^{k+1})_{i_k}\}.$$
It can be seen that the pseudoinverse  solution $A^\dagger b$ belongs to  $H_{i_k}$, and $\tilde{x}^{k+1}$ is the orthogonal projection of $\tilde{x}^k$ onto $H_{i_k,\tilde{z}^{k+1}}$.
We use $\tilde{x}^{k+1}_*$ to denote the orthogonal projection of $\tilde{x}^k$ onto $H_{i_k}$. In fact, $\tilde{x}^{k+1}$ can now be regarded as an approximation of $\tilde{x}^{k+1}_*$.  The geometric interpretation of REK is presented in Figure \ref{figureREK}.


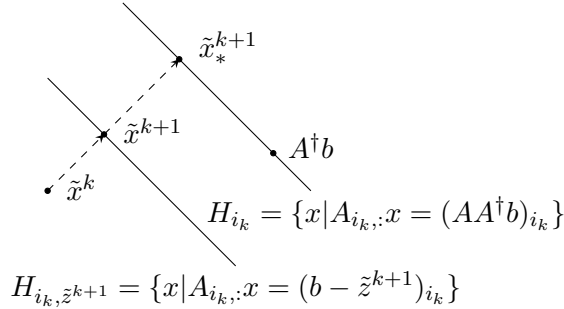
\begin{figure}[hptb]
 	\centering
	\begin{tikzpicture}
	\draw (1.5,3.5)--(4,1);
	\draw (0.5,2.5)--(3,0);
	
	\filldraw (2.25,2.75) circle [radius=1pt]
                      (1.25,1.75) circle [radius=1pt]
                      (0.5,1) circle [radius=1pt]
                      (3.5,1.5) circle [radius=1pt];
            \draw (2.9,2.9) node {$\tilde{x}^{k+1}_*$};
            \draw (1.9,1.8) node {$\tilde{x}^{k+1}$};
            \draw (0.95,1.05) node {$\tilde{x}^k$};
            \draw (4,1.6) node {$A^{\dag}b$};
            \draw (5,0.7) node {$H_{i_k}=\{x| A_{i_k, :} x=(AA^{\dag}b)_{i_k}\}$};
            \draw (3,-0.3) node {$H_{i_k,\tilde{z}^{k+1}}=\{x|A_{i_k, :} x=(b-\tilde{z}^{k+1})_{i_k}\}$};
            \draw [dashed,-stealth] (0.5,1) -- (1.25,1.75);
            \draw [dashed,-stealth] (1.25,1.75)--(2.25,2.75);

         \end{tikzpicture}
         \caption{The geometric interpretation of REK \cite{Zou12,Du19}. The next iterate $\tilde{x}^{k+1}$ can be regarded as an approximation of $\tilde{x}^{k+1}_*$.}	
         \label{figureREK}
\end{figure}

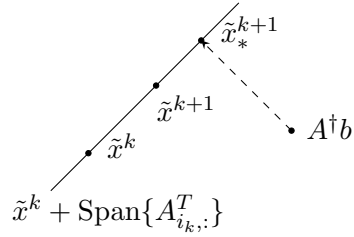
\begin{figure}[hptb]
 	\centering
	\begin{tikzpicture}
	\draw (0,0)--(2.5,2.5);
	
	\filldraw (2,2) circle [radius=1pt]
                      (1.4,1.4) circle [radius=1pt]
                      (0.5,0.5) circle [radius=1pt]
                      (3.2,0.8) circle [radius=1pt];
            \draw (2.65,2.1) node {$\tilde{x}_*^{k+1}$};
            \draw (1.8,1.1) node {$\tilde{x}^{k+1}$};
            \draw (0.95,0.6) node {$\tilde{x}^k$};
            \draw (3.7,0.85) node {$A^{\dag}b$};
            \draw (0.9,-0.3) node {$\tilde{x}^k+\mbox{Span}\{A_{i_k,:}^T\}$};
            \draw [dashed,-stealth] (3.2,0.8)--(2,2);

         \end{tikzpicture}
         \caption{The geometric interpretation of REK from a different perspective.}	
         \label{figureREKnew}
\end{figure}

To understand the  relationship between REK and RKAS, we shall examine the geometric interpretation of REK from a different perspective. Indeed, $\tilde{x}_*^{k+1}$ can be also regarded as the orthogonal projection of $A^{\dag}b$ onto the space $\tilde{x}^k+\mbox{Span}\{A_{i_k,:}^T\}$, see Figure \ref{figureREKnew}. At each step,  REK seeks a point $\tilde{x}^{k+1}$ that belongs to $\tilde{x}^k+\mbox{Span}\{A_{i_k,:}^T\}$ and approximately minimizes  $\|\tilde{x}^{k+1}-A^\dagger b\|^2_2$, since finding the optimal $\tilde{x}_*^{k+1}$ may be difficult in practice.
This means that REK can be regarded as an error-minimizing method to some extent, or an \emph{inexact error-minimizing} method.
For the RKAS method, it follows from \eqref{opt-prob} that at each step,  we find an $x^{k+1}$ belonging to $x^k+\mbox{Span}\{A_{i_k,:}^T\}$ such that  $\left\|Ax^{k+1}-AA^\dagger b\right\|^2_2$ is minimized, which implies that RKAS can be regarded as a \emph{residual-minimizing} method.

It can be observed from \eqref{rek} that the iterates $\tilde{x}^{k}$ obtained by REK can also be interpreted as a variation of RK with adaptive stepsizes, where the sequence $\{\tilde{z}^{k}\}_{k=0}^\infty$ is an auxiliary variable used to update the stepsizes.
 As shown by Du \cite[Theorem 2]{Du19}, the convergence factor for REK is $1-\frac{\sigma^2_{\min}(A)}{\|A\|^2_F}$, which is better than $1-\frac{\sigma^4_{\min}(A)}{\|A\|^2_2\|A\|^2_F}$ established in Corollary \ref{corol-1}.
 However, in our numerical experiments, we find that RKAS is better than REK for handling large-scale sparse  systems.

For the sparse matrix $A\in\mathbb{R}^{m\times n}$, assume that all of its rows and columns are not
equal to zero. Moreover, we assume that its $j$-th column, i.e.  $A_{:, j}(j=1,\ldots,n)$, has $m_j$ nonzero entries.  Let us define
$$
T=\left\{(i, \ell) \mid A_{i, :} \circ A_{\ell, :} \neq 0, 1 \leq i, \ell \leq m \right\}
$$
and for any fixed $i\in[m]$,
$$
T_i=\left\{(i, \ell) \mid A_{i, :} \circ A_{\ell, :} \neq 0, 1 \leq \ell \leq m\right\},
$$
where $\circ$ denotes the Hadamard product. From the definition, we have
$$
T=\cup_{i=1}^m T_i, \ \text{and } \ T_i \cap T_j=\emptyset, \forall i\neq j.
$$
We note that since $A$ is a sparse matrix, the cardinality of $T$, i.e. $\mid T\mid$, can be much smaller than $m^2$, and in practical issues sometimes $O(m)$.
For any $(i,\ell) \in T$, we assume that $A_{i, :} \circ A_{\ell, :}$ has $s_{i,\ell}$ nonzero entries. Therefore, we know that
the $i$-th column of $A$ now has $s_{i,i}$ nonzero entries.

For the REK method, its initialization for computing
$\|A_{i, :}\|_2^2$ $(i =1,\ldots, m)$ and $\|A_{:, j}\|_2^2$ $(j=1, \ldots, n)$ costs
$$
2 \sum\limits_{i=1}^m s_{i, i}+2\sum\limits_{j=1}^n m_j-m-n
$$
flops, and its execution of the $k$-th iterate costs
$$
4s_{i_k, i_k}+4m_{j_k}+2
$$
flops. For the RKAS method, if we store $AA^\top$ at the beginning, then
the initialization of the RKAS method for computing $AA^\top$, $\|AA_{i,:}\|_2^2$, and $\|A_{i, :}\|_2^2$ $(i=1,\ldots,m)$ costs
$$
\sum_{(i, \ell) \in T, i \leq \ell} (2s_{i,\ell}-1)+2 \sum\limits_{i=1}^m s_{i, i}+\frac{3}{2} \mid T\mid-\frac{3m}{2}
$$
flops, and the execution of the $k$-th iterate of the RAKS method costs
$$
2s_{i_k, i_k}+4 \mid T_{i_k}\mid
$$
flops. If we do not store $AA^\top$ at the initialization,
then the  initialization of  the RKAS for computing
$\|A_{i, :}\|_2^2$ $(i =1,\ldots, m)$  costs
$$
2 \sum\limits_{i=1}^m s_{i, i}-m
$$
flops, and the execution of the $k$-th iterate of the RKAS method costs
$$
2s_{i_k, i_k}+5 \mid T_{i_k}\mid+2\sum_{\ell \in T_{i_k}} s_{i_k, \ell}-1
$$
flops.

\section{Numerical experiments}
In this section, we describe some numerical results for the RKAS method for inconsistent systems. We also compare RKAS with REK \cite{Zou12,Du19} on a variety of test problems.
 All methods are implemented in  {\sc Matlab} R2022a for Windows $10$ on a desktop PC with the  Intel(R) Core(TM) i7-10710U CPU @ 1.10GHz  and 16 GB memory.

As in Du et al \cite{Du20Ran}, to construct an inconsistent linear system, we set $b=A x+r$, where $x$ is a vector with entries generated from a standard normal distribution and the residual $r \in \operatorname{Null}\left(A^\top\right)$. Note that one can obtain such a vector $r$ by the {\sc Matlab} function {\tt null}. For RKAS, we set $x^0=0$ and store $AA^\top$ at the initialization, and for REK, we set $z^0=b$ and $x^0=0$. We stop the algorithms if the relative solution error (RSE)
$\frac{\|x^k-A^\dagger b\|^2_2}{\|A^\dagger b\|^2_2}\leq10^{-12}$.
We report the average number of iterations (denoted as Iter) and the average computing time in seconds (denoted as CPU) of RKAS and REK.

\subsection{Synthetic data}
We use the following two types of coefficient matrices.

\begin{itemize}
  \item For given $m, n, r$, and $\kappa>1$, we construct a dense matrix $A$ by $A=U D V^T$, where $U \in \mathbb{R}^{m \times r}, D \in \mathbb{R}^{r \times r}$, and $V \in \mathbb{R}^{n \times r}$. Using {\sc Matlab} colon notation, these matrices are generated by {\tt [U,$\sim$]=qr(randn(m,r),0)}, {\tt [V,$\sim$]=qr(randn(n,r),0)}, and {\tt D=diag(1+($\kappa$-1).*rand(r,1))}. So the condition number of $A$ is upper bounded by $\kappa$.
  \item We construct a random sparse matrix by using the {\sc Matlab} sparse random matrix function {\tt sprandn(m,n,density,rc)}, where {\tt density} is the percentage of nonzero entries and {\tt rc} is the reciprocal of the condition number.
\end{itemize}

Figures \ref{figue1} and \ref{figue2} illustrate our experimental results with a fixed $n$.
In Figure \ref{figue1}, we plot the computing time of the REK and RKAS for inconsistent linear systems with coefficient matrices $A=UDV^\top$, where $m=1000,2000,\ldots,10000,n=100,r=80$, $\kappa(A)=2$ (left) or $\kappa(A)=10$ (right).
It can be observed from Figure \ref{figue1} that REK is more efficient than RKAS for solving the dense problem, and the changing of parameter $\kappa(A)$ affects the performance of RKAS greatly than that of REK.

In Figure \ref{figue2}, we plot the computing time of the REK and RKAS with random sparse matrices $A$, where $m=1000,2000,\ldots,10000,n=100$,$\kappa(A)=2$ (left) or $\kappa(A)=10$ (right), and the sparsity of the coefficient matrices $A$ is $0.1$. Noting that now $\kappa(A)$ is exactly the condition number of  $A$. It can be seen that RKAS performs better than REK.
 It can be also found that the more rows there are than columns, the better RKAS performs than REK. This is due to REK adopting both a row and a column at each step, while RKAS is a row-action method where only a single row is used at each step.
To illustrate this observation more clearly, in Figure \ref{figue3}, we plot the computing times of REK and RKAS with a fixed $m=10000$ and $n=100,200,\ldots,1000$.  It is clear that the performance of RKAS is better than REK when $n$ is small, and  REK is better than RKAS when $n$ is large. 

\begin{figure}[hptb]
	\centering
	\begin{tabular}{cc}
\includegraphics[width=0.33\linewidth]{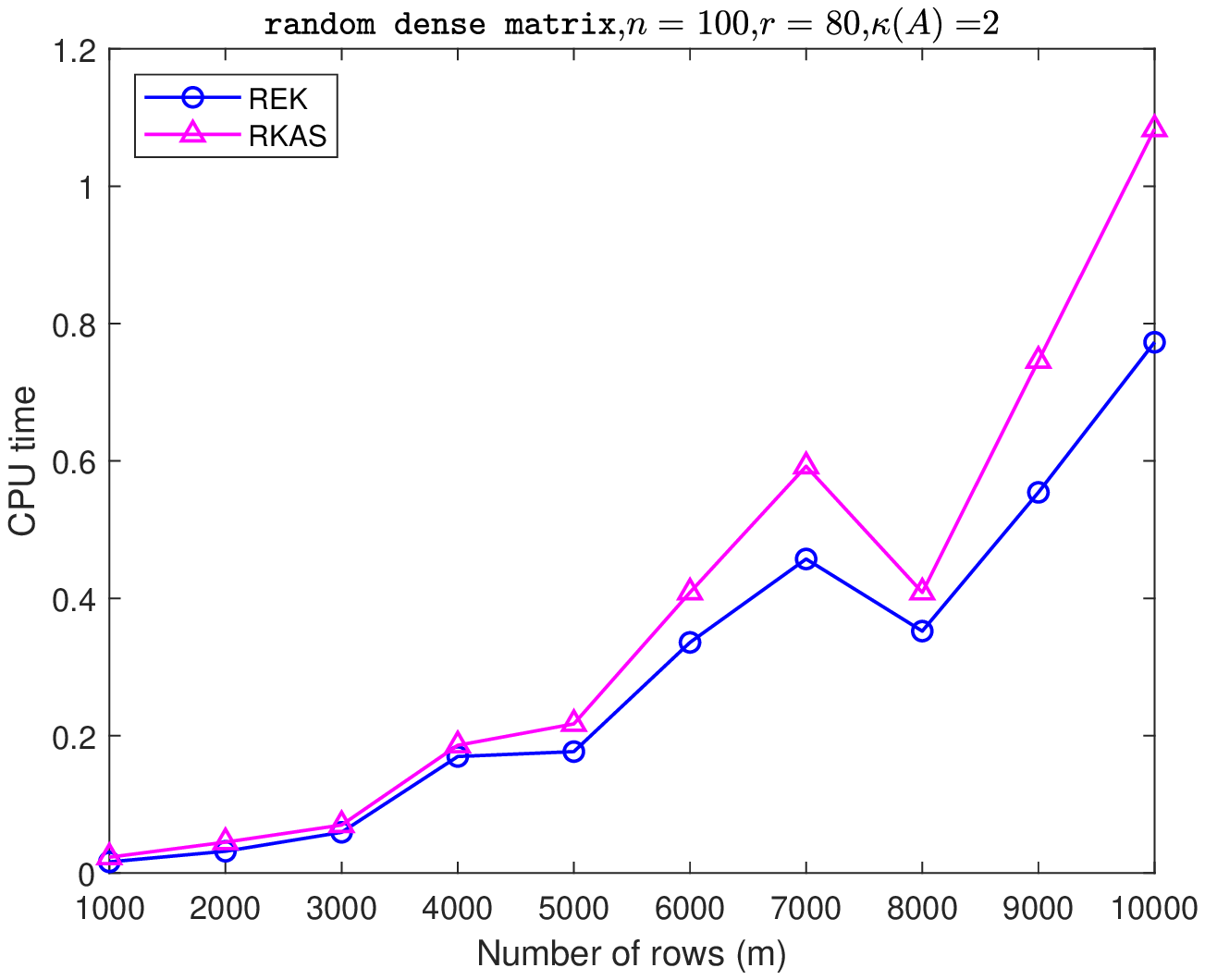}
\includegraphics[width=0.33\linewidth]{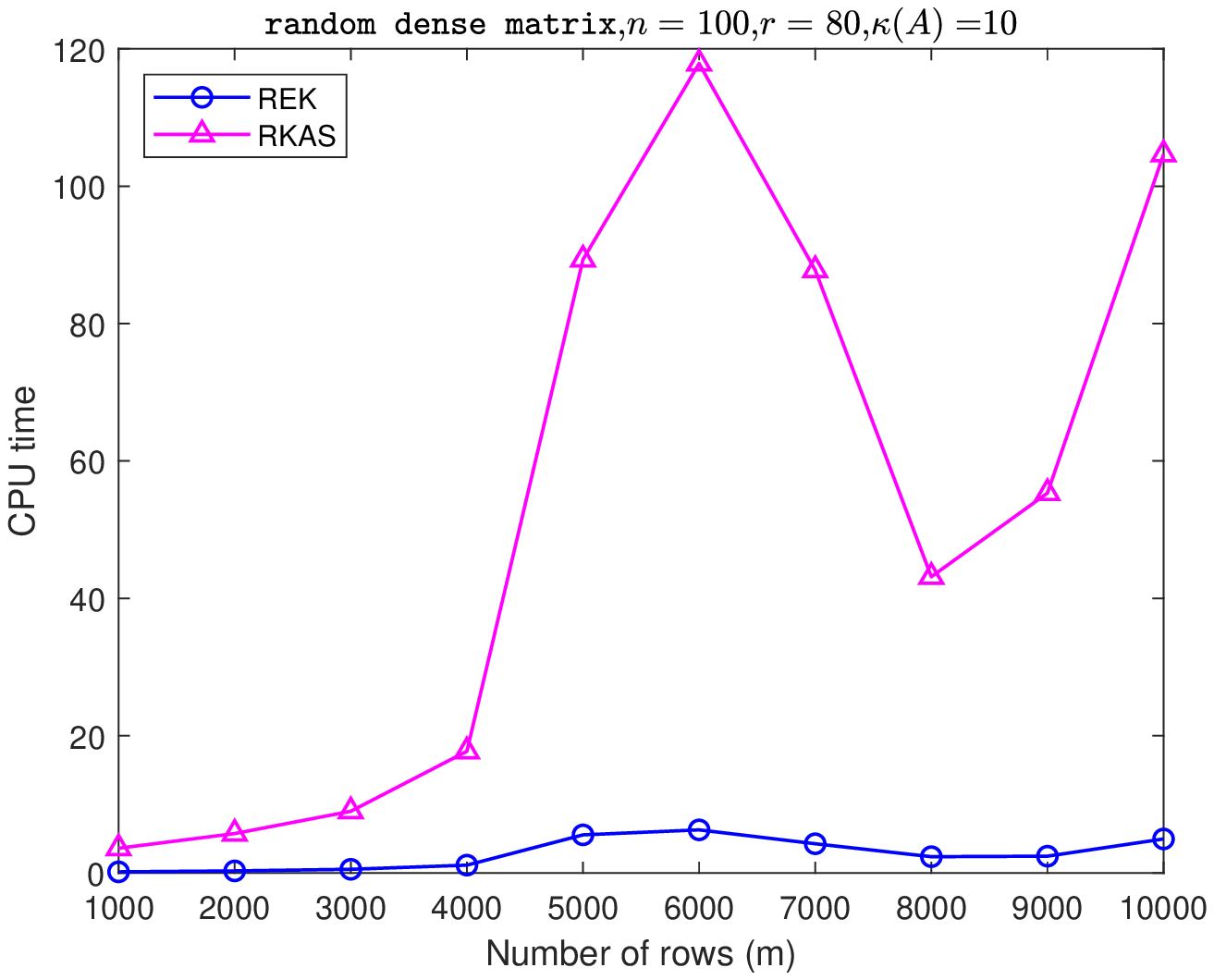}
	\end{tabular}
	\caption{Figures depict the CPU time (in seconds) vs increasing number of rows for the case of the random dense matrix.  The title of each plot indicates the values of $n,r$, and $\kappa$. All plots are averaged over 50 trials. }
	\label{figue1}
\end{figure}

\begin{figure}[hptb]
	\centering
	\begin{tabular}{cc}
\includegraphics[width=0.33\linewidth]{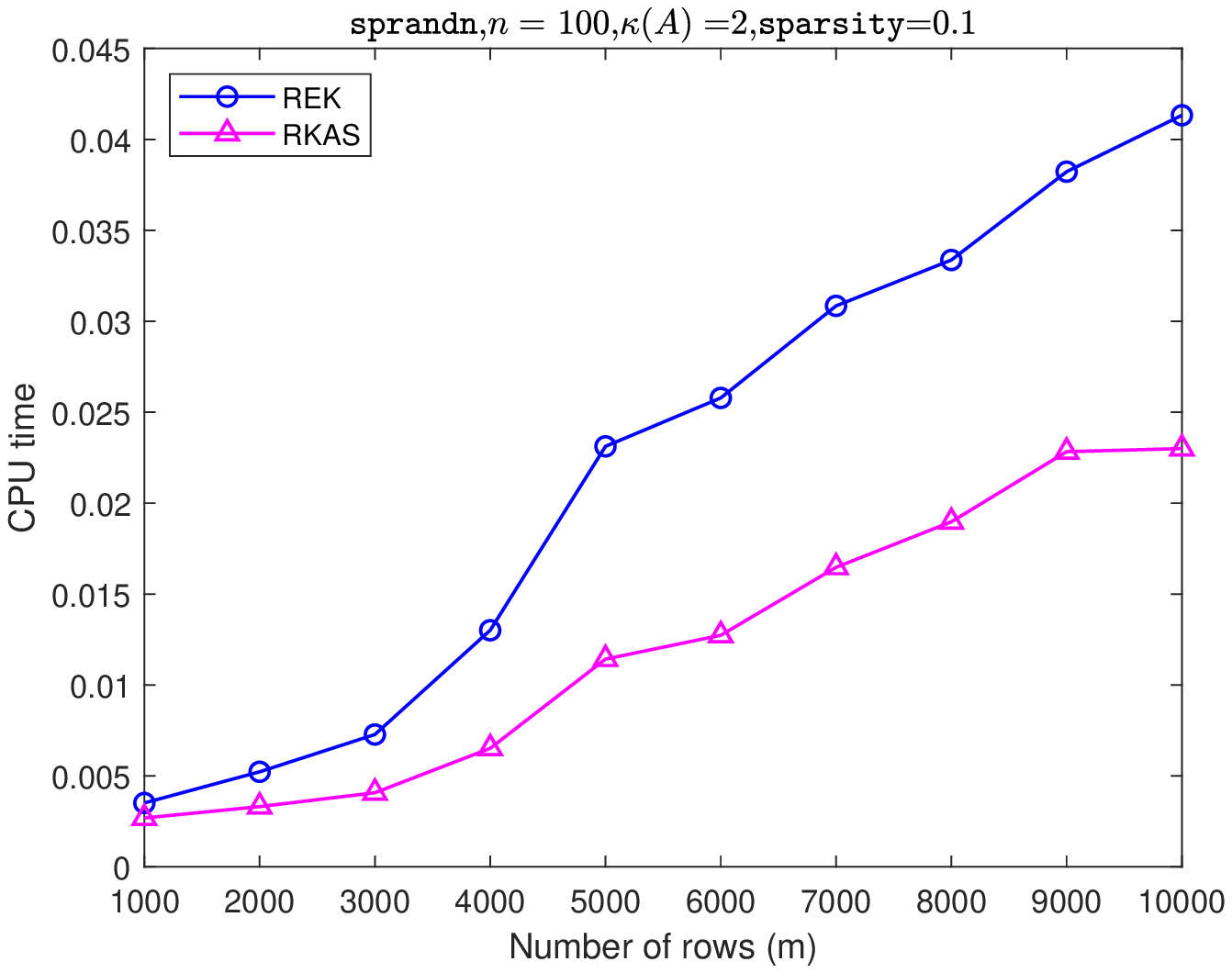}
\includegraphics[width=0.33\linewidth]{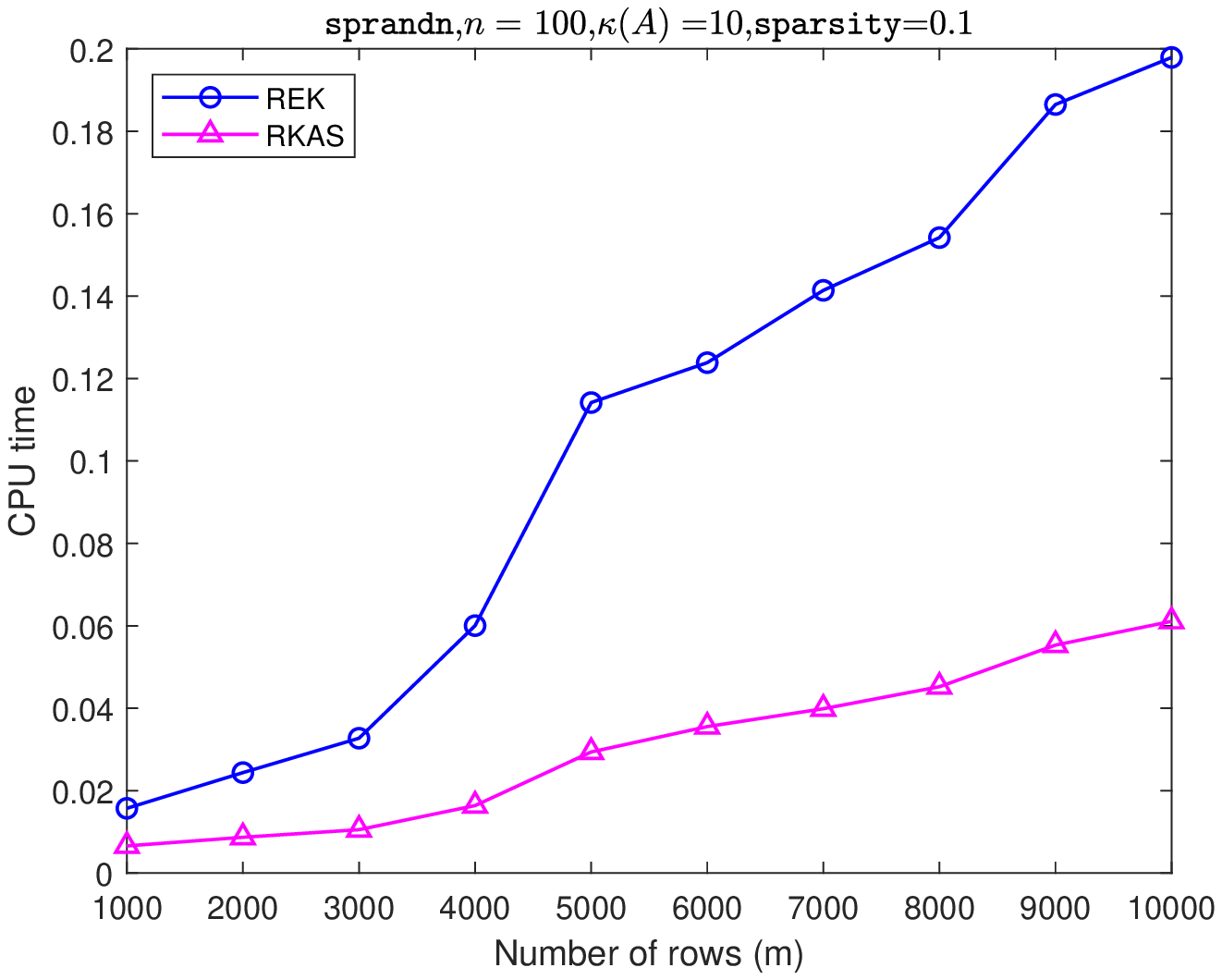}
	\end{tabular}
	\caption{Figures depict the CPU time (in seconds) vs increasing number of rows for the case of the random sparse matrix.  The title of each plot indicates the values of $n$, $\kappa$, and sparsity. All plots are averaged over 50 trials.  }
	\label{figue2}
\end{figure}

\begin{figure}[hptb]
	\centering
	\begin{tabular}{cc}
\includegraphics[width=0.33\linewidth]{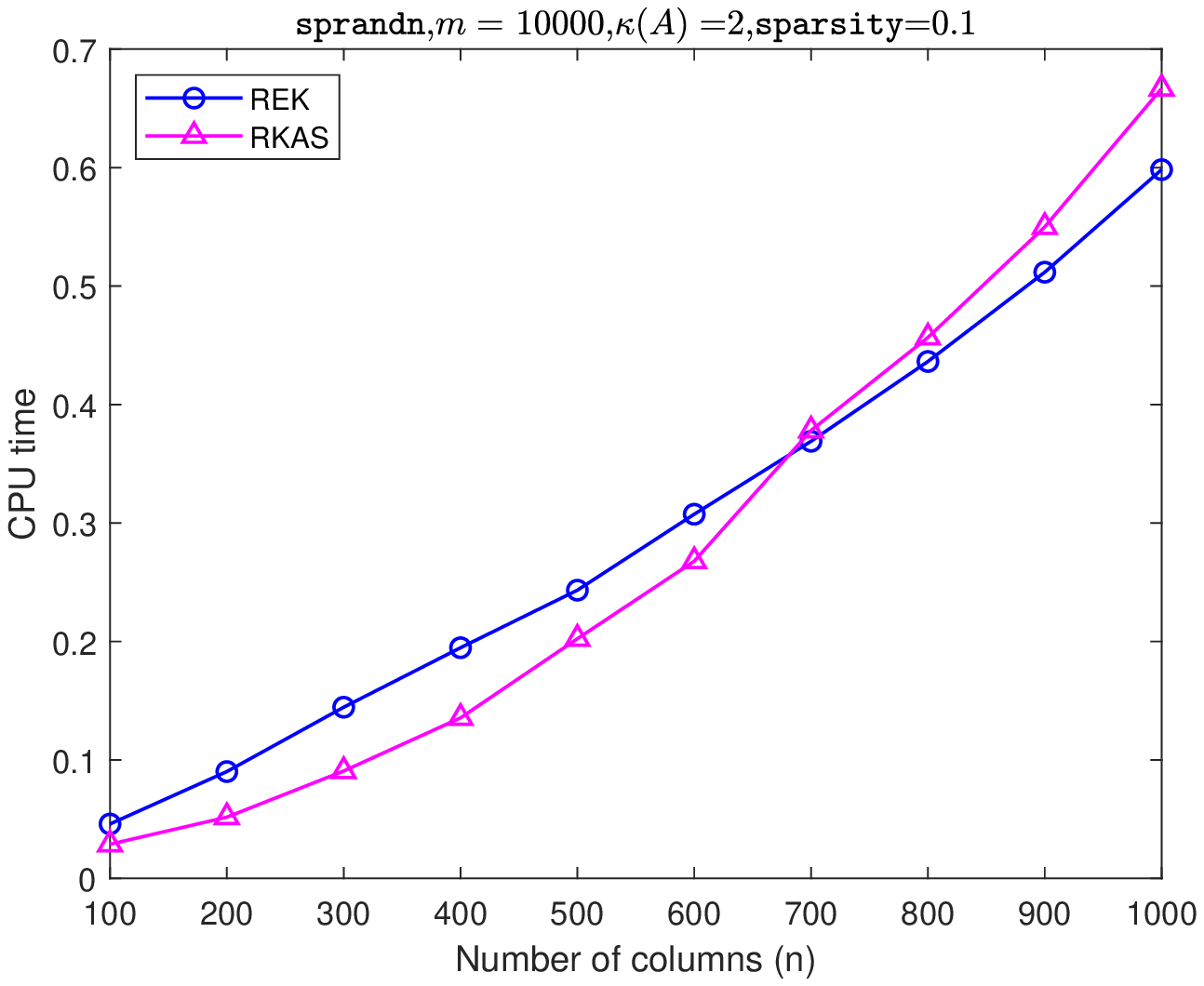}
\includegraphics[width=0.33\linewidth]{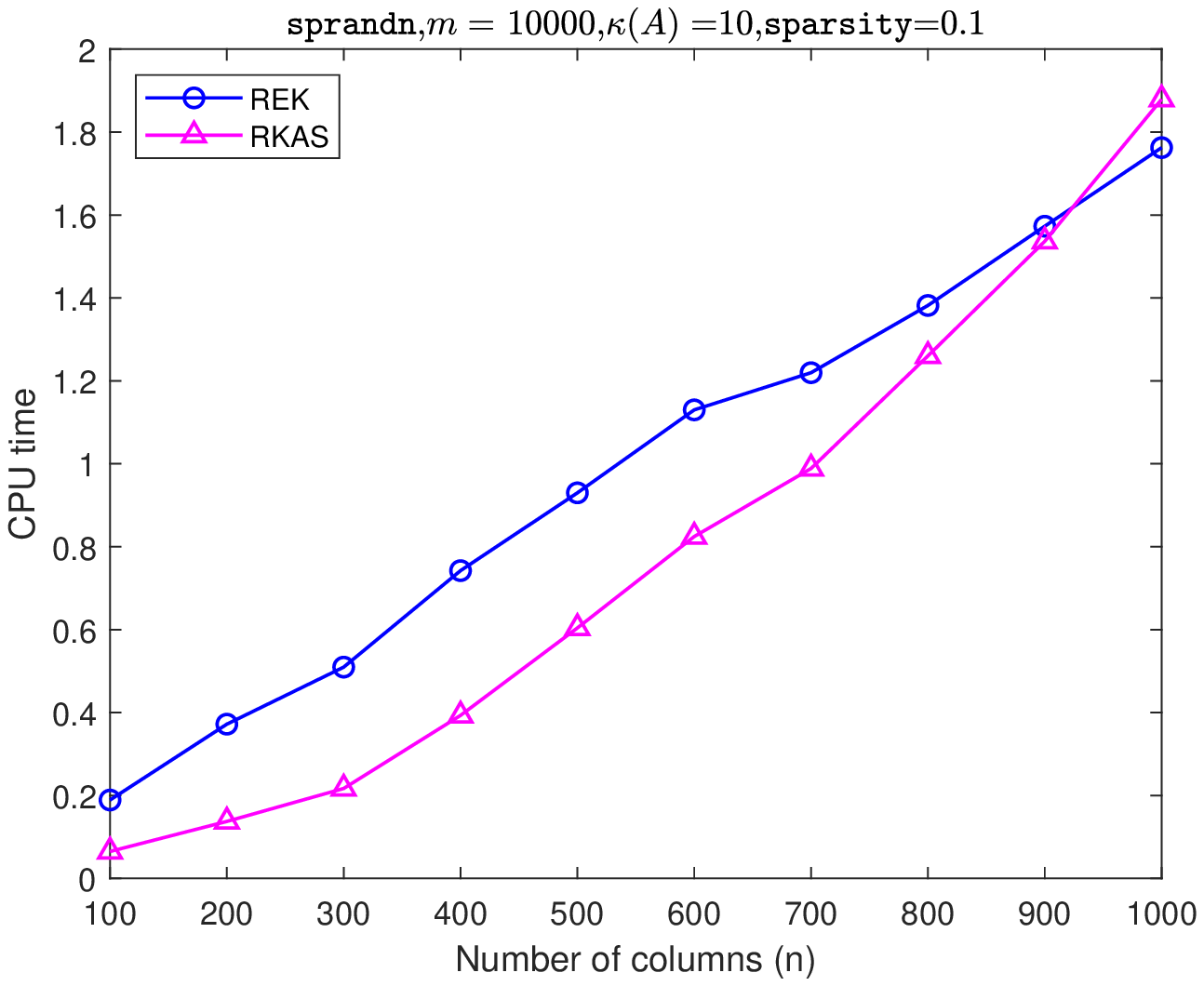}
	\end{tabular}
	\caption{Figures depict the CPU time (in seconds) vs increasing number of columns for the case of the random sparse matrix.  The title of each plot indicates the values of $m$, $\kappa$, and sparsity. All plots are averaged over 50 trials.  }
	\label{figue3}
\end{figure}

\subsection{Real-world data} The real-world data are available via the SuiteSparse Matrix Collection \cite{Kol19}. The five matrices are {\tt nemsafm}, {\tt df2177}, {\tt ch8\_8\_b1}, {\tt bibd\_16\_8}, and {\tt ash958}.
Each dataset consists of a matrix $A\in\mathbb{R}^{m\times n}$ and a vector $b\in\mathbb{R}^m$. In our experiments, we only use the matrices $A$ of the datasets and ignore the vector $b$. In Table \ref{table1}, we report the number of iterations and the computing times for  REK and RKAS. It can be observed that  RKAS is comparable with REK for solving inconsistent linear systems.

\begin{table}
\renewcommand\arraystretch{1.5}
\setlength{\tabcolsep}{2pt}
\caption{ The average (50 trials of each algorithm) Iter and CPU of REK and RKAS for inconsistent linear systems with coefficient matrices from \cite{Kol19}.}
\label{table1}
\centering
{\scriptsize
\begin{tabular}{  |c| c| c| c| c |c |c |c|  }
\hline
\multirow{2}{*}{ Matrix}& \multirow{2}{*}{ $m\times n$ }  &\multirow{2}{*}{rank}& \multirow{2}{*}{$\frac{\sigma_{\max}(A)}{\sigma_{\min}(A)}$}  &\multicolumn{2}{c| }{REK}  &\multicolumn{2}{c| }{RKAS}
\\
\cline{5-8}
& &   &    & Iter & CPU    & Iter & CPU      \\
\hline
 {\tt nemsafm} & $334\times 2348$ &   334  &4.77  &   41308.70  &   {\bf 1.3104} & 120565.48  &   2.3087  \\
\hline
{\tt df2177} & $630\times10358$ &  630 & 2.01  &   20192.62  &   5.1010 & 21480.34  &   {\bf 2.9148}   \\
\hline
{\tt ch8\_8\_b1} & $1568\times64$ &   63  & 3.48e+14  &   1800.96  &   0.0186 &  1686.84  &   {\bf 0.0136}  \\
\hline
{\tt bibd\_16\_8}& $120\times12870$ &  120  & 9.54 &    7859.60  &   {\bf 3.3143} & 151632.30  &  32.4403  \\
\hline
{\tt ash958} & $958\times292$ &  292  &3.20  &     15711.02  &   {\bf 0.1037} & 42197.00  &   0.1924   \\
\hline
\end{tabular}
}
\end{table}

\section{Concluding remarks}

Consider the following the least-squares problem
\begin{equation}\label{quar-f}
\min\limits_{x\in\mathbb{R}^n}f(x):=\frac{1}{2 m}\|A x-b\|_2^2=\frac{1}{m} \sum_{i=1}^m f_i(x),
\end{equation}
where $f_i(x)=\frac{1}{2}\left(A_{i,:}x-b_i\right)^2$.
To state conveniently, we assume that $A \in \mathbb{R}^{m \times n}$ is normalized to $\|A_{i,:}\|^2_2=1$ for each row of $A$.
The RK method \eqref{RK-method} can be seen as \emph{stochastic gradient descent} (SGD) \cite{hardt2016train,robbins1951stochastic,ma2017stochastic} applied to the least-squares problem \eqref{quar-f}. Indeed,  SGD solves \eqref{quar-f} using unbiased estimates for the gradient of the objective function, i.e.  $\nabla f_i(x)$ such that $\mathbb{E}\left[\nabla f_i(x)\right]=\nabla f(x)$. At each iteration, a random unbiased estimate $\nabla f_i(x)$ is drawn and SGD uses the following update formula
\begin{equation}\label{SGD}
x^{k+1}=x^k-\lambda_k \nabla f_i(x^k),
\end{equation}
where $\lambda_k$ is an appropriately chosen stepsize. Noting that if a random row of the matrix $A$ is selected and \eqref{SGD} is computed with $\nabla f_i(x^k)=\left( A_{i,:}x^k-b_i\right)A_{i,:}^\top$, then one can recover the RK method.

It is well-known that SGD suffers from slow convergence as the variance of the gradient estimate  $\nabla f_i(x)$ does not naturally diminish, i.e. $\lim\limits_{k \rightarrow \infty}\mathbb{E}[\|\nabla f_{i_k} (x^k)-\nabla f(x^k)\|^2_2] \neq 0$.
Let $x_*$ be an optimal point of \eqref{quar-f} and consider the variance of its gradient estimate
$$
\sigma^2
=\mathbb{E}[\| \nabla f_i(x_*)-\nabla f(x_*)\|^2_2]
=\mathbb{E}[\| \nabla f_i(x_*)\|^2_2]
=\sum\limits_{i=1}^m \frac{(A_{i,:}x_*  - b_i)^2}{m}
=\frac{\|e\|^2_2}{m},
$$
where $e=Ax_*-b$ is the residual at $x_{*}$.
When the system is consistent, as the iterate approaches $x_*$, the residual gradually drops to zero and thus so does the variance, which ensures the convergence of SGD with a constant stepsize.
When the system is inconsistent, however, $e \neq 0$ and the variance does not decrease to zero. In this case, variance reduction techniques are introduced \cite{nguyen2017sarah,johnson2013accelerating}, otherwise a \textit{decreasing stepsize} is required. Nevertheless, the decreasing stepsize brings about adverse effect on the convergence of SGD, which is sublinear even if the objective function is strongly convex \cite{nemirovski2009robust}.
In this paper, we have shown that RK, i.e. SGD for \eqref{quar-f}, with our adaptive stepsize strategy enjoys a linear rate without any variance reduction procedure.
A natural extension of our results is the design and analysis of adaptive stepsizes for SGD in the case of general convex or strongly convex functions. This should be an interesting and valuable topic that deserves in-depth study in the future.

Finally, we note that a bunch of advanced probability criteria have been investigated in the literature for the RK method, such as the greedy selection rule \cite{Bai18Gre} and  the weighted version \cite{Ste20Wei}. These criteria are convenient to be adapted to the RKAS context for further improvement in efficiency.

\bibliographystyle{abbrv}
\bibliography{main1230}

\end{document}